\documentclass[12pt,twoside]{amsart}
\usepackage{latexsym,amsmath,amsopn,amssymb,amsthm,amsfonts}
\usepackage[T2A]{fontenc}
\usepackage[cp1251]{inputenc}
\usepackage[english]{babel}
\usepackage{graphicx}

\newtheorem{theorem}{Theorem}

\newtheorem{proposition}{Proposition}

\newtheorem{remark}{Remark}

\newtheorem{corollary}{Corollary}

\newtheorem{lemma}{Lemma}

\begin{document}

	\title[sub-Finsler Cartan group]{Extremals of a left-invariant sub-Finsler quasimetric on the Cartan group}
	\author{V.~N.~Berestovskii, I.~A.~Zubareva}
	\thanks{The work is supported by Mathematical Center in Akademgorodok under agreement No. 075-15-2019-1613 with 
		the	Ministry of Science and Higher Education of the Russian Federation}
	\address{Sobolev Institute of Mathematics, \newline
		Acad. Koptyug avenue, 4, Novosibirsk, 630090, Russia}
	\email{vberestov@inbox.ru}
	\address{Sobolev Institute of Mathematics,\newline
		Pevtsova str., 13, Omsk, 644099, Russia}
	\email{i\_gribanova@mail.ru}	
	\begin{abstract}
		Using the Pontryagin Maximum Principle for the time-optimal problem in coordinates of the first kind, we find extremals of abitrary left--invariant sub--Finsler quasimetric on the Cartan group defined by a distribution of rank two.
		
		\vspace{2mm}
		
	\end{abstract}
	\maketitle
	
	\section*{Introduction}
	
	In \cite{Ber1}, it is indicated that the shortest arcs of any left-invariant (sub-)Finsler metric $d$ on a Lie group $G$ are
	solutions of a left-invariant time-optimal problem with the closed unit ball $U$ of some {\it arbitrary} norm $F$ on a
	subspace $\mathfrak{p}$ of the Lie algebra $(\mathfrak{g},[\cdot,\cdot])$ of the Lie group $G$ as a control region. 
	In addition, the subspace $\mathfrak{p}$ generates $\mathfrak{g}$. These statements are also valid for (sub-)Finsler quasimetrics and the corresponding quasinorms. We explain that quasimetric have all properties of metric, except possibly the symmetry property $d(p,q)=d(q,p).$ Moreover, $U$ is an arbitrary convex figure in $\mathfrak{p}$ with $0$ interior to $U$, perhaps $U\neq -U$.
	The Pontryagin Maximum Principle  \cite{PBGM} gives the necessary conditions for
	optimal trajectories of the problem;  the curves, satisfying these conditions, are called {\it extremals}. 
	Apparently, for the first time the shortest arcs of any left-invariant sub-Finsler metric on Lie group have been found in paper \cite{Ber2} in the case of arbitrary sub-Finsler metric $d$ on the Heisenberg group $H$.
	
	In this paper we find extremals of arbitrary left-invariant sub-Finsler quasimetric on the Cartan group, defined by a subspace  $\mathfrak{p}$ of rank two; every extremal is normal for corresponding control. In papers \cite{S} by Sachkov and \cite{ALDS}, \cite{ALDS1} by Ardentov, Le Donne, Sachkov, they considered special cases in other coordinates.
	
	We apply here classical methods and results from the monograph \cite{PBGM}. Paper \cite{BerZub} uses some new search methods for normal extremals of left-invariant (sub-)Finsler and (sub-)Riemannian metrics.
	
	The authors thank S.~K.~Vodopyanov for useful discussions. 
	
	\section{The Campbell-Hausdorff formula for the Cartan group}
	
	Let $X$, $Y$, $Z$, $V$, $W$ be a basis of the five-dimensional Cartan algebra $\mathfrak{g}$ such that
	\begin{equation}
	\label{a1}
	[X,Y]=Z,\quad [X,Z]=V,\quad [Y,Z]=W,
	\end{equation}
	all other Lie brackets are equal to zero. 
	Thus $\mathfrak{g}$ is a nilpotent Lie algebra with two generators $X$, $Y$.
	Therefore, as it is known, there exists a unique up to isomorphism connected simply connected nilpotent Lie group
	$G$ with the Lie algebra $\mathfrak{g}$, the Cartan group, and the exponential mapping $\exp: \mathfrak{g}\rightarrow G$ is a diffeomorphism. This diffeomomorphism and the Cartesian coordinates $x,\,y,\,z,\,v,w$ in  $\mathfrak{g}$  with the basis $X$, $Y$, $Z$, $V$, $W$ defines coordinates of the first kind on $G$ and thus a diffeomorphism $G\cong\mathbb{R}^5$.
	
	\begin{proposition}
		\label{product}
		In coordinates of the first kind, the multiplication on the Cartan group $G\cong\mathbb{R}^5$ is given by 
		the following rule	
		\begin{equation}
		\label{a2}
		\left(\begin{array}{c}
		x_1 \\
		y_1 \\
		z_1 \\
		v_1 \\
		w_1
		\end{array}\right)\times\left(\begin{array}{c}
		x_2 \\
		y_2 \\
		z_2 \\
		v_2 \\
		w_2
		\end{array}\right)=\left(\begin{array}{c}
		x_1+x_2 \\
		y_1+y_2 \\
		z_1+z_2+\frac{1}{2}(x_1y_2-x_2y_1) \\
		v_1+v_2+\frac{1}{2}(x_1z_2-x_2z_1)+\frac{1}{12}(x_1^2y_2-x_1x_2y_1-x_1x_2y_2+x_2^2y_1) \\
		w_1+w_2+\frac{1}{2}(y_1z_2-y_2z_1)+\frac{1}{12}(x_1y_1y_2+x_2y_1y_2-x_2y_1^2-x_1y_2^2)
		\end{array}\right).
		\end{equation}
	\end{proposition}
	
	\begin{proof}
		Set $A_i=x_iX+y_iY+z_iZ+v_iV+w_iW$, $i=1,2$.
		Using (\ref{a1}), we consequently obtain
		$$[A_1,A_2]=(x_1y_2-x_2y_1)Z+(x_1z_2-x_2z_1)V+(y_1z_2-y_2z_1)W;$$
		$$[A_1,[A_1,A_2]]=x_1(x_1y_2-x_2y_1)V+y_1(x_1y_2-x_2y_1)W;$$
		$$[A_2,[A_2,A_1]]=[[A_1,A_2],A_2]=-x_2(x_1y_2-x_2y_1)V-y_2(x_1y_2-x_2y_1)W.$$
		
		Since the Lie algebra $\mathfrak{g}$ is of step three, then it is valid the following Campbell-Hausdorff formula
		(see \cite{Post}):
		$$\ln\left(\exp(A_1)\exp(A_2)\right)=A_1+A_2+\frac{1}{2}[A_1,A_2]+\frac{1}{12}[A_1,[A_1,A_2]]+\frac{1}{12}[A_2,[A_2,A_1]].$$
		Therefore
		$$\ln\left(\exp(A_1)\exp(A_2)\right)=(x_1+x_2)X+ (y_1+y_2)Y+\left(z_1+z_2+\frac{1}{2}(x_1y_2-x_2y_1)\right)Z+$$
		$$\left(v_1+v_2+\frac{1}{2}(x_1z_2-x_2z_1)+\frac{1}{12}(x_1^2y_2-x_1x_2y_1-x_1x_2y_2+x_2^2y_1)\right)V+$$
		$$\left(w_1+w_2+\frac{1}{2}(y_1z_2-y_2z_1)+\frac{1}{12}(x_1y_1y_2+x_2y_1y_2-x_2y_1^2-x_1y_2^2)\right)W.$$
		The last equality gives (\ref{a2}).
	\end{proof}
	
	It follows from the applied method to introduce coordinates of the first kind and formulas (\ref{a2}) that in these coordinates,
	the chosen basis of the Lie algebra $\mathfrak{g}$ is realized as left-invariant vector fields on the Lie group $G$ of the form
	\begin{equation}
	\label{XY}
	X= \frac{\partial}{\partial x} -\frac{y}{2}\frac{\partial}{\partial z} - \frac{z}{2}\frac{\partial}{\partial v} -\frac{xy}{12}\frac{\partial}{\partial v}-\frac{y^2}{12}\frac{\partial}{\partial w},\,\,
	Y= \frac{\partial}{\partial y} + \frac{x}{2}\frac{\partial}{\partial z} + \frac{x^2}{12}\frac{\partial}{\partial v}
	+\left(\frac{xy}{12}-\frac{z}{2}\right)\frac{\partial}{\partial w},
	\end{equation}
	\begin{equation}
	\label{ZVW}
	Z= \frac{\partial}{\partial z} + \frac{x}{2}\frac{\partial}{\partial v} +\frac{y}{2}\frac{\partial}{\partial w},\,\,
	V = \frac{\partial}{\partial v},\,\,W = \frac{\partial}{\partial w}.
	\end{equation}

	\section{Left-invariant sub-Finsler quasimetric and the optimal control on the Cartan group}
	
	In \cite{Ber1}, it is said that  the shortest arcs of a left-invariant sub-Finsler metric $d$ on arbitrary connected Lie group	$G$
	defined by a left-invariant bracket generating distribution $D$ and a norm  $F$ on $D(e)$ coincide with the time-optimal solutions of
	the following control system 
	\begin{equation}
	\label{a3}
	\dot{g}(t)=dl_{g(t)}(u(t)),\quad u(t)\in U,
	\end{equation}
	with measurable controls $u=u(t)$. Here $l_g(h)=gh$,  the control region is the unit ball  $$U=\{u\in D(e)\,|F(u)\leq 1\}.$$
	This statement is also true in the case when $d$ is a quasimetric (respectively, $F$ is a quasinorm on $D(e)$).
	
	Therein the Pontryagin Maximum Principle \cite{PBGM} for (local) time optimal control $u(t)$ and corresponding trajectory  $g(t),$ $t\in\mathbb{R}$ implies the existence of a non-vanishing absolutely continuous vector-function
	$\psi(t)\in T^{\ast}_{g(t)}G$ such that for almost all $t\in\mathbb{R}$ the function 
	$\mathcal{H}(g(t);\psi(t);u)=\psi(t)(dl_{g(t)}(u))$ of the variable $u\in U$ attains a maximum at the point $u(t)$: 
	\begin{equation}
	\label{m}
	M(t)=\psi(t)(dl_{g(t)}(u(t)))=\max\limits_{u\in U}\psi(t)(dl_{g(t)}(u)).
	\end{equation}
	In addition, the function $M(t),$  $t\in\mathbb{R},$ is constant and non-negative, $M(t)\equiv M\geq 0$.
	In case when $M=0$ (respectively, $M>0$) the corresponding  {\it extremal}, i.e. the curve, satisfying the Pontryagin Maximum Principle, is called {\it abnormal} (respectively, {\it normal}).
	
	If $x=(x^1,\dots, x^n)$ is a global coordinate system on $G,$ 
	$$x(t)=(x^1(t),\dots x^n(t)):=(x^1(g(t)),\dots x^n(g(t))),$$
	$$\psi_j=\psi_j(t)=\psi(t)\left(\frac{\partial}{\partial x^j}\right)(x(t)), j=1,\dots,n,\quad \psi(t):=(\psi_1(t),\dots,\psi_n(t)),$$ 
	then according to \cite{PBGM}, the pair $(g(t),\psi(t))$ satisfies {\it the Hamiltonian system} in a symbolic notation
	\begin{equation}
	\label{ham}
	\dot{x}(t)=\frac{\partial\mathcal{H}}{\partial\psi}(x(t),\psi(t),u(t)),\quad \dot{\psi}(t)=
	-\frac{\partial\mathcal{H}}{\partial x}(x(t),\psi(t),u(t)).
	\end{equation}
	
	It follows from (\ref{a1}) that the left-invariant distribution $D$ on $G$ with the basis $X, Y$ for $D(e)$ is bracket generating.
	Let $F$ be an arbitrary quasinorm on $D(e)$. Then the pair  $(D(e),F)$ defines a left-invariant sub-Finsler quasimetric  
	$d$ on $G$; therein $u_1X(e)+u_2Y(e)$ is identified with $u=(u_1,u_2),$ where $u_i\in\mathbb{R},$ $i=1,2.$
	
	Let $\psi_k,$ $k=1,\dots,5,$ be covector components of $\psi=\psi(t)$ relative to the coordinate system $(x,y,z,v,w),$  i.e.
	\begin{equation}
	\label{psic}
	\psi_1=\psi\left(\frac{\partial}{\partial x}\right),\,\,\psi_2=\psi\left(\frac{\partial}{\partial y}\right),\,\,
	\psi_3=\psi\left(\frac{\partial}{\partial z}\right),\,\,\psi_4=\psi\left(\frac{\partial}{\partial v}\right),\,\,
	\psi_5=\psi\left(\frac{\partial}{\partial w}\right),
	\end{equation}
	\begin{equation}
	\label{hc}
	h_1=\psi(X),\quad h_2=\psi(Y),\quad h_3=\psi(Z),\quad h_4=\psi(V),\quad h_5=\psi(W).
	\end{equation}
	
	Using (\ref{XY}), (\ref{ZVW}), (\ref{psic}), (\ref{hc}), we obtain 
	\begin{equation}
	\label{a5}
	h_1=\psi_1-\frac{1}{2}\psi_3y-\frac{1}{12}\psi_4xy-\frac{1}{12}\psi_5y^2-\frac{1}{2}\psi_4z,\quad h_2=\psi_2+\frac{1}{2}\psi_3 x+\frac{1}{12}\psi_4 x^2+\frac{1}{12}\psi_5xy-\frac{1}{2}\psi_5z,
	\end{equation}
	\begin{equation}
	\label{h34}
	h_3= \psi_3 + \psi_4\frac{x}{2}+\psi_5\frac{y}{2},\quad h_4=\psi_4,\quad h_5=\psi_5.
	\end{equation}
	
	Then the function $\mathcal{H}(x,y,z,v,w;\psi_1,\psi_2,\psi_3,\psi_4,\psi_5;u_1,u_2)$ can be written as
	\begin{equation}
	\label{ph}
	\mathcal{H}=\psi(u_1X+u_2Y)=u_1\psi(X)+u_2\psi(Y)=h_1u_1+h_2u_2.
	\end{equation}
	
	With regard to the first equality in (\ref{ham}), (\ref{ph}) and (\ref{a5}) system (\ref{a3}) takes a form
	\begin{equation}
	\label{a4}
	\dot{x}(t)=u_1,\quad\dot{y}(t)=u_2,\quad\dot{z}(t)=\frac{1}{2}(xu_2-yu_1),
	\end{equation}
	\begin{equation}
	\label{a4_1}
	\dot{v}(t)=-\frac{1}{2}\left(z+\frac{1}{6}xy\right)u_1+\frac{1}{12}x^2u_2,\quad
	\dot{w}(t)=-\frac{1}{12}y^2u_1-\frac{1}{2}\left(z-\frac{1}{6}xy\right)u_2,
	\end{equation}
	where $(u_1,u_2)=(u_1(t),u_2(t))\in U$. 
	
	{\it In consequence of left-invariance of the metric $d$ we can assume that the trajectories initiate at the unit $e\in G$, i.e. $x(0)=y(0)=z(0)=v(0)=w(0)=0$.}
	
	The control $u=u(t)=(u_1(t),u_2(t))\in U,$ $t\in \mathbb{R},$ defined by the Pontryagin Maximum Principle is bounded and measurable  \cite{PBGM},
	therefore integrable. Then the functions $x(t),$ $y(t),$ $t\in\mathbb{R},$ defined by the first two equations in (\ref{a4}) are Lipschitz, the product of any finite number	of these functions is Lipschitz, and its derivative is bounded and measurable on each compact segment of $\mathbb{R}$. So this derivative can be computed by the usual differentiation rule of a product from differential calculus for functions of one variable. Therefore, the last  equation of the system (\ref{a4}) and equations of (\ref{a4_1}) can be integrated by parts, using the first two equations in (\ref{a4}) (see ss. 2.9.21, 2.9.24 in \cite{Fed}). By $x(0)=y(0)=z(0)=v(0)=w(0)=0$
	we get successively
	\begin{equation}
	\label{aa4}
	z(t)=-\frac{1}{2}x(t)y(t)+\int\limits_0^t x(\tau)u_2(\tau)d\tau,
	\end{equation}
	\begin{equation}
	\label{aa5}
	v(t)=\frac{1}{12}x^2(t)y(t)-\frac{1}{2}x(t)\int\limits_0^t x(\tau)u_2(\tau)d\tau+\frac{1}{2}\int\limits_0^tx^2(\tau)u_2(\tau)d\tau,
	\end{equation}
	\begin{equation}
	\label{aa6}
	w(t)=-\frac{1}{12}x(t)y^2(t)-\frac{1}{2}y(t)\int\limits_0^t x(\tau)u_2(\tau)d\tau+\int\limits_0^tx(\tau)y(\tau)u_2(\tau)d\tau.
	\end{equation}
	
	By (\ref{ph}) and (\ref{a5}), the second equality in (\ref{ham}) defines the following ODE system 
	conjugate to (\ref{a4}), (\ref{a4_1}), for the absolutely continuous vector function $\psi=\psi(t)$: 
	\begin{equation}
	\label{a6}
	\left\{\begin{array}{l}
	\dot{\psi_1}=\frac{1}{12}\psi_4yu_1-\left(\frac{1}{2}\psi_3+\frac{1}{6}\psi_4x+\frac{1}{12}
	\psi_5y\right)u_2,\\
	\dot{\psi_2}=\left(\frac{1}{2}\psi_3+\frac{1}{12}\psi_4x+\frac{1}{6}\psi_5y\right)u_1-\frac{1}{12}\psi_5xu_2,\\
	\dot{\psi_3}=\frac{1}{2}\psi_4u_1+\frac{1}{2}\psi_5u_2,\\
	\quad\dot{\psi_4}=0,\\
	\quad\dot{\psi_5}=0.
	\end{array}\right.
	\end{equation}
	Assign an arbitrary set of initial data $\psi_i(0)=\varphi_i$, $i=1,\dots,5$, of the system (\ref{a6}). It follows from (\ref{a6}), (\ref{a4}) and the initial condition  $x(0)=y(0)=0$ that
	\begin{equation}
	\label{psi3}
	\psi_5\equiv\varphi_5,\quad\psi_4\equiv\varphi_4,\quad\psi_3=\varphi_3+\frac{1}{2}\varphi_4x+\frac{1}{2}\varphi_5y.
	\end{equation}
	Notice that $\left(\frac{1}{2}xy+z\right)^{\cdot}=xu_2$, 
	$\left(\frac{1}{2}xy-z\right)^{\cdot}=yu_1$ on the ground of (\ref{a4}). With regard
	to (\ref{psi3}) the first and the second equations in (\ref{a6}) take a form
	$$\dot{\psi}_1=\frac{1}{12}\varphi_4\left(\frac{1}{2}xy-z\right)^{\cdot}-\frac{1}{2}\varphi_3\dot{y}-\frac{5}{12}\varphi_4\left(\frac{1}{2}xy+z\right)^{\cdot}-\frac{1}{3}\varphi_5y\dot{y},$$
	$$\dot{\psi}_2=\frac{5}{12}\varphi_5\left(\frac{1}{2}xy-z\right)^{\cdot}+
	\frac{1}{2}\varphi_3\dot{x}-
	\frac{1}{12}\varphi_5\left(\frac{1}{2}xy+z\right)^{\cdot}+\frac{1}{3}\varphi_4x\dot{x}.$$
	Therefore, by the initial data of systems (\ref{a4}) and (\ref{a6}), we get
	\begin{equation}
	\label{psi12}
	\psi_1=\varphi_1-\frac{1}{2}\varphi_3y-\frac{1}{6}\varphi_5y^2-\frac{1}{6}\varphi_4xy-
	\frac{1}{2}\varphi_4z,\,\,
	\psi_2=\varphi_2+\frac{1}{2}\varphi_3x+\frac{1}{6}\varphi_4x^2+\frac{1}{6}\varphi_5xy-\frac{1}{2}\varphi_5z.
	\end{equation}
	Inserting (\ref{psi3}) and (\ref{psi12}) into (\ref{a5}), (\ref{h34}), we find
	\begin{equation}
	\label{h12}
	h_1=\varphi_1-\left(\varphi_3+\frac{1}{2}\varphi_4x+\frac{1}{2}\varphi_5y\right)y-\varphi_4z,\quad h_2=\varphi_2+\left(\varphi_3 +\frac{1}{2}\varphi_4 x+\frac{1}{2}\varphi_5y\right)x-\varphi_5z,
	\end{equation}
	
	\begin{equation}
	\label{h345}
	h_3=\varphi_3 + \varphi_4x + \varphi_5y,\quad h_4=\varphi_4,\quad h_5=\varphi_5.
	\end{equation}
		
	From (\ref{h12}) and (\ref{h345}) we obtain an integral of the Hamiltonian system (\ref{a4}) -- (\ref{a4_1}), (\ref{a6}):
		\begin{equation}
		\label{c12}
		\mathcal{E}=\frac{h_3^2}{2}+h_1h_5-h_2h_4\equiv\frac{\varphi_3^2}{2}+\varphi_1\varphi_5-\varphi_2\varphi_4.
		\end{equation}
Thus the functions $\mathcal{H}(t)=M(t)$ and three the so-called {\it Casimir functions} $h_4=\varphi_4,$ $h_5=\varphi_5,$ and $\mathcal{E}$ are integrals of this Hamiltonian system.
    
Now, using (\ref{a4}), (\ref{h12}) and (\ref{h345}), we compute
\begin{equation}
\label{a7}
\dot{h}_1=-h_3u_2,\quad \dot{h}_2=h_3u_1.
\end{equation}

For an extremal $\left(x(t),y(t),z(t),v(t),w(t)\right),$ a bounded measurable control $u(t)$ and 
a non-vanishing absolutely continuous vector-function $\psi(t)$,  the function \linebreak
$\mathcal{H}(x(t),y(t),z(t),v(t),w(t);\psi_1(t),\psi_2(t),\psi_3(t),\psi_4(t),\psi_5(t);u_1,u_2)$ of 
$u\in U$ attains the maximum at the point $u=u(t)$:
\begin{equation}
\label{m}
M(t)=h_1(t)u_1(t)+h_2(t)u_2(t)=\max\limits_{u\in U}(h_1(t)u_1+h_2(t)u_2)\equiv M\geq 0.
\end{equation}

Relations (\ref{a4}), (\ref{h12}) and (\ref{m}) imply that under multiplication of functions $\psi_i(t)$, $i=1,\dots,5$, by a positive constant $k$ the trajectory $(x(t),y(t),z(t),v(t),w(t))$ does not change, while $M$ is multipled by $k$. Therefore {\it in case when $M>0$ we shall assume that} $M=1$. 
{\it Further in this section we consider this case.}

It follows from (\ref{m}) that $(h_1(t),h_2(t))$ in (\ref{h12}) and $(\varphi_1,\varphi_2)=(h_1(0),h_2(0))$ lie
on the boundary $\partial U^{\ast}$ of the polar figure $U^{\ast}=\{h\,|F_{U}(h)\leq 1\}$ to
$U$, where $F_{U}$ is a quasinorm on $H=\{h\},$ is equal to the support Minkowski function of the body $U$:
$$F_U(h)=\max\limits_{u\in U}h\cdot u.$$
In addition, $(H,F_U)$ is the conjugate quasinormed vector space to $(D(e),F)$ and  $(U^{\ast})^{\ast}=U$ (see Theorem 14.5 in \cite{Rock}).
Moreover, using (\ref{a7}) and (\ref{m}), we get
\begin{equation}
\label{area}
h_1(t)\dot{h}_2(t)-\dot{h}_1(t)h_2(t)=h_3(t)(h_1(t)u_1(t)+h_2(t)u_2(t))=h_3(t).
\end{equation}

Let $r=r(\theta)$, $\theta\in\mathbb{R}$, be a polar equation of the curve $F_U(x,y)=1$.  At every point $\theta\in \mathbb{R}$
there exist one-sided derivatives of $r=r(\theta)$ (and with possible exclusion of no more than countable number of values
$\theta$ there exists the usual derivative $r'(\theta)$). For simplicity {\it we shall denote every value between these derivatives by} $r'(\theta)$. Then 
\begin{equation}
\label{h}
h_1(t)=h_1(\theta)=r(\theta)\cos\theta,\quad h_2(t)=h_2(\theta)=r(\theta)\sin\theta,\quad \theta=\theta(t),
\end{equation}   
\begin{equation}
\label{hder}  
h'_1(\theta)=-(r(\theta)\sin\theta-r^{\prime}(\theta)\cos\theta),\quad h'_2(\theta)=(r^{\prime}(\theta)\sin\theta+r(\theta)\cos\theta).
\end{equation}

Independently on the existence of usual derivative (\ref{hder}),  (\ref{area}) implies the existence of usual derivative 
for the doubled oriented area   
$$\sigma(t)=2S(\theta(t))=\int_0^{\theta(t)}r^2(\theta)d\theta$$ 
of the sector, counted from $0.$ In addition, by (\ref{h34}) and (\ref{area}) 
\begin{equation}
\label{derst}
\dot{\sigma}(t)=\varphi_3+\varphi_4x(t)+\varphi_5y(t)=r^2(\theta(t))\dot{\theta}(t),\quad \dot{\theta}(t)=\frac{\dot{\sigma}(t)}{r^2(\theta(t))}.
\end{equation}
If we square the second equality in (\ref{derst}), we get by (\ref{h12})
$$r^4(\theta)\dot{\theta}^2=\varphi_3^2+\left(\varphi_3+\frac{1}{2}\varphi_4x+
\frac{1}{2}\varphi_5y\right)\left(2\varphi_4+2\varphi_5y\right)=$$
$$\varphi_3^2+2\varphi_4(h_2-\varphi_2)-2\varphi_5(h_1-\varphi_1),$$
\begin{equation}
\label{dt}
\dot{\theta}^2=\frac{\varphi_3^2+2\varphi_4(h_2-\varphi_2)-2\varphi_5(h_1-\varphi_1)}{r^4(\theta)}.
\end{equation}
On the ground of (\ref{c12}), (\ref{area}), and (\ref{derst}),
\begin{equation}
\label{pen}
\ddot{\sigma}(t)=\varphi_4u_1(t)+\varphi_5u_2(t),
\end{equation}
\begin{equation}
\label{energy}
\mathcal{E}=\mathcal{E}(t)=\frac{1}{2}(\dot{\sigma}(t))^2+h_1(t)h_5(t)-h_2(t)h_4(t)= 
\end{equation}
$$\frac{1}{2}(h_3(t))^2+h_1(t)h_5(t)-h_2(t)h_4(t)= \frac{\varphi_3^2}{2}+\varphi_1\varphi_5-\varphi_2\varphi_4.$$

\begin{remark}
(\ref{energy}) is equivalent to (\ref{dt}).	
\end{remark}

It follows from (\ref{a4}) and (\ref{a4_1}) that
\begin{equation}
\label{t1}
\left(3v+\frac{1}{2}xz\right)^{\cdot}=-\frac{3}{2}\dot{x}z+\frac{1}{2}x\dot{z}+\frac{1}{2}\dot{x}z+\frac{1}{2}x\dot{z}=x\dot{z}-\dot{x}z,
\end{equation}
\begin{equation}
\label{t2}
\left(3w+\frac{1}{2}yz\right)^{\cdot}=-\frac{3}{2}\dot{y}z+\frac{1}{2}y\dot{z}+
\frac{1}{2}\dot{y}z+\frac{1}{2}y\dot{z}=y\dot{z}-\dot{y}z,
\end{equation}
so on the base of  (\ref{a4}), (\ref{h12}), and (\ref{m})  we get, omitting for brevity the variable $t$,
$$h_1u_1+h_2u_2=\varphi_1\dot{x}+\varphi_2\dot{y}+2\varphi_3\dot{z}+\varphi_4(x\dot{z}-z\dot{x})+\varphi_5(y\dot{z}-z\dot{y})=$$
$$\left(\varphi_1x+\varphi_2y+2\varphi_3z+3\varphi_4v+3\varphi_5w+\frac{\varphi_4}{2}xz+
\frac{\varphi_5}{2}yz\right)^{\cdot}=1.$$
Taking into account of the initial data of systems (\ref{a4}) and (\ref{a4_1}), we obtain
\begin{equation}
\label{eq}
\varphi_1x(t)+\varphi_2y(t)+2\varphi_3z(t)+3\varphi_4v(t)+3\varphi_5w(t)+\frac{\varphi_4}{2}x(t)z(t)+
\frac{\varphi_5}{2}y(t)z(t)=t.
\end{equation}

\section{Search for sub-Finsler extremals}

\label{geod}

{\bf 1.} Let us consider an abnormal case. The following proposition is valid.

\begin{proposition}
	\label{ageod1}
	An abnormal extremal  $(x,y,z,v,w)(t)$, $t\in\mathbb{R}$, on the Cartan group starting at the unit is one  
	of the following one-parameter subgroups
	\begin{equation}
	\label{anorm1}
	x(t)\equiv 0,\quad y(t)=\frac{st}{F(0,s)},\quad s=\pm 1,\quad z(t)=v(t)=w(t)\equiv 0,
	\end{equation}
	\begin{equation}
	\label{anorm2}
	x(t)=\frac{st}{F(s,0)}\quad s=\pm 1, \quad y(t)=z(t)=v(t)=w(t)\equiv 0,
	\end{equation}
	\begin{equation}
	\label{anorm}
	x(t)=\frac{s\varphi_5t}{F(s\varphi_5,-s\varphi_4)},\,\,y(t)=\frac{-\varphi_4x(t)}{\varphi_5}, s=\pm 1,
	z(t)=v(t)=w(t)\equiv 0\neq \varphi_4\cdot\varphi_5,
	\end{equation}
	and is not strongly abnormal.
\end{proposition}

\begin{proof}
Assume that $M=0$. Then we obtain from the maximum condition that $h_1(t)=h_2(t)\equiv 0$ and $\varphi_1=\varphi_2=0$. Since $u_1(t)$ and $u_2(t)$ could not simultaneously vanish at any  $t\in\mathbb{R}$, then $\varphi_3+\varphi_4x(t)+\varphi_5y(t)\equiv 0$ on the base of (\ref{h345}) and (\ref{a7}). 
This implies that $\varphi_3=0$ and $\varphi_4x(t)+\varphi_5y(t)\equiv 0$  because $x(0)=y(0)=0$.  Hence in consequence of (\ref{psi3}) and (\ref{psi12})
we get $\varphi_4\neq 0$ or/and $\varphi_5\neq 0$ because $\psi(t)$ does not vanish. It follows from this and (\ref{h12}) that $z(t)\equiv 0$.

Let $\varphi_4\neq 0$, $\varphi_5=0$. Then $x(t)\equiv 0$ and  $u_1(t)\equiv 0$ according to the first equation (\ref{a4}). 
Hence in consequence of (\ref{a4_1}) and the initial condition $v(0)=w(0)=0$ we successively get $v(t)=w(t)\equiv 0$.
Further, since $u_1(t)\equiv 0$ and $F(u_1(t),u_2(t))\equiv 1$, then $u_2(t)\equiv\frac{s}{F(0,s)},$ $s=\pm 1$.
This, the second equation in  (\ref{a4}), and the initial condition $y(0)=0$ imply that $y(t)=\frac{st}{F(0,s)}$, $s=\pm 1,$ and we get (\ref{anorm1}).
In consequence of (\ref{a2}), the extremal is one of two one-parameter subgroups 
$$g(t)=\exp\left(\frac{stY}{F(0,s)}\right),\quad s=\pm 1,\quad t\in\mathbb{R},$$ 
satisfies (\ref{m}) with $M(t)\equiv 1$  for constant covector function
$$\psi(t)=(0,\varphi_2,0,0,0)=(0,sF(0,s),0,0,0)=(0,h_2(t),0,0,0),\quad s=\pm 1,$$
subject to differential equations (\ref{a6}) and (\ref{a7}); therefore it is normal relative to this covector function and is not strongly abnormal.

Let $\varphi_4=0$, $\varphi_5\neq 0$. Then $y(t)\equiv 0$ and  $u_2(t)\equiv 0$ by the second equation (\ref{a4}). Then from  (\ref{a4}), (\ref{a4_1}) and the initial condition $v(0)=w(0)=0$ we successively get $v(t)=w(t)\equiv 0$.
Also, since $u_2(t)\equiv 0$ and $F(u_1(t),u_2(t))\equiv 1$ then $u_1(t)\equiv\frac{s}{F(s,0)},$ $s=\pm 1$.
This, the first equation in  (\ref{a4}), and the initial condition $x(0)=0$ imply that $x(t)=\frac{st}{F(s,0)}$, $s=\pm 1,$ and we get (\ref{anorm2}). 
In consequence of (\ref{a2}), the extremal is one of two one-parameter subgroups 
$$g(t)=\exp\left(\frac{stX}{F(s,0)}\right),\quad s=\pm 1,\quad t\in\mathbb{R},$$ 
satisfies (\ref{m}) with $M(t)\equiv 1$  for constant covector function
$$\psi(t)=(\varphi_1,0,0,0,0)=(s F(s,0),0,0,0,0)=(h_1(t),0,0,0,0),\quad s=\pm 1,$$
subject to differential equations (\ref{a6}) and (\ref{a7}); therefore it is normal relative to this covector function and is not strongly abnormal.  

Let $\varphi_4\neq 0$ and $\varphi_5\neq 0$. Then $u_2(t)=-\frac{\varphi_4}{\varphi_5}u_1(t)$ on the ground of (\ref{a4}) and the equality $\varphi_4x(t)+\varphi_5y(t)\equiv 0$. Since $F(u_1(t),u_2(t))\equiv 1$ then $u_1(t)\equiv\frac{s\varphi_5}{F(s\varphi_5,-s\varphi_4)}$, $s=\pm 1.$  
This, (\ref{a4}), and the  initial condition $x(0)=z(0)=0$ imply that $x(t)=\frac{s\varphi_5t}{F(s\varphi_5,-s\varphi_4)},$ $z(t)\equiv 0$. 
By substitution the equalities $y(t)=-\frac{\varphi_4}{\varphi_5}x(t)$, $u_2(t)=-\frac{\varphi_4}{\varphi_5}u_1(t),$ and $z(t)\equiv 0$ to the equations (\ref{a4_1}), we get $\dot{v}(t)=\dot{w}(t)\equiv 0$, whence $v(t)=w(t)\equiv 0$ because of $v(0)=w(0)=0$.
In consequence of (\ref{a2}), the extremal is one of two one-parameter subgroups
 $$g(t)=\exp\left(\frac{st(\varphi_5 X-\varphi_4 Y)}{F(s\varphi_5,-s\varphi_4)}\right),\quad s=\pm 1,\quad t\in\mathbb{R},$$ 
satisfies (\ref{m}) with $M(t)\equiv 1$ for constant covector function
 $$\psi(t)=(\varphi_1,\varphi_2,0,0,0)=$$ $$\left(\frac{F(s\varphi_5,-s\varphi_4)}{2\varphi_5},-\frac{F(s\varphi_5,-s\varphi_4)}{2\varphi_4},0,0,0\right)=(h_1(t),h_2(t),0,0,0),s=\pm 1,$$
 subject to differential equations (\ref{a6}) and (\ref{a7}); therefore it is normal relative to this covector function and is not strongly abnormal.
\end{proof}

{\bf 2.} Set $M=1$. 

\begin{theorem}
	\label{mt}	
	For every extremal on the Cartan group starting at the unit, 
	\begin{equation}
	\label{xxt}
	x(t)=\int_0^t\frac{[r^{\prime}(\theta(\tau))\sin\theta(\tau)+r(\theta(\tau))\cos\theta(\tau)]d\tau}{r^2(\theta(\tau))},
	\end{equation} 
	\begin{equation}
	\label{yyt}
	y(t)=\int_0^t\frac{[r(\theta(\tau))\sin\theta(\tau)-r^{\prime}(\theta(\tau))\cos\theta(\tau)]d\tau}{r^2(\theta(\tau))}
	\end{equation}
	with arbitrary measureable integrands of indicated view and continuously differentiable function 
	$\theta=\theta(t),$ satisfying  (\ref{derst}), (\ref{dt}). The functions $z(t),$ $v(t),$ $w(t)$ are defined by formulae 
	(\ref{aa4}), (\ref{aa5}), (\ref{aa6}) or
	\begin{equation}
	\label{zet}
	z(t)=\frac{1}{2}\int_0^t(x\dot{y}-y\dot{x})d\tau,
	\end{equation}
	\begin{equation}
	\label{vw}
	v(t)=\frac{1}{3}\int\limits_0^t(x\dot{z}-z\dot{x})d\tau-\frac{1}{6}x(t)z(t),\quad
	w(t)=\frac{1}{3}\int\limits_0^t(y\dot{z}-z\dot{y})d\tau-\frac{1}{6}y(t)z(t).
	\end{equation}   	
\end{theorem}

\begin{proof}
By Proposition \ref{ageod1}, every extremal is normal for corresponding control.  
The proof of the first statement  is completed as in the theorem 1 in \cite{BerZub}. The equalities (\ref{zet}), (\ref{vw})
are consequences of (\ref{a4}), (\ref{t1}), (\ref{t2}) and the initial condition $z(0)=v(0)=w(0)=0$.	
\end{proof}

{\bf 2.1.} Let us assume that $\varphi_3=\varphi_4=\varphi_5=0$. The following proposition is true.

\begin{proposition}
	\label{norm1}
	For any extremal on the Cartan group with conditions $\varphi_3=\varphi_4=\varphi_5=0$ and the unit origin, $\theta(t)\equiv\theta_0,$ $t\in\mathbb{R},$ for some $\theta_0.$ In addition, every such extremal is a one-parameter subgroup if and only if there exists the usual derivative $r'(\theta_0).$ In general case, any extremal with conditions 
	$\varphi_3=\varphi_4=\varphi_5=0$ is a metric straight line.
\end{proposition}

\begin{proof}
	The first statement follows from (\ref{derst}). 	
	
	In addition, by Theorem \ref{mt}, every admissible control $(u_1(t),u_2(t))=(u_1(\theta_0),u_2(\theta_0)),$ with components equal to the integrands in 	(\ref{xxt}), (\ref{yyt}), is constant if and only if there exists the usual derivative $r'(\theta_0),$ what is equivalent to
	condition that the system  (\ref{a4})--(\ref{a4_1}) has unique solution, a one-parameter subgroup
	$$x(t)=u_1(\theta_0)t,\quad y(t)=u_2(\theta_0)t,\quad z(t)=v(t)=w(t)\equiv 0.$$
		
	Notice that there exists at most countable number of values $\theta_0$  for which the second statement is false. 
	For any such $\theta_0,$ $x(t),$ $y(t),$ $t\in\mathbb{R},$ are as in (\ref{xxt}), (\ref{yyt})
	with $\theta(\tau)\equiv \theta_0$ and arbitrary measurable integrands  $u_1(\tau),$ $u_2(\tau)$
	of the type, indicated in Theorem \ref{mt}, and the functions $z(t)$, $v(t)$ and $w(t)$  are defined by formulas (\ref{aa4}), (\ref{aa5}) and
	(\ref{aa6}) respectively.

	It follows from (\ref{a4}) that the length of any arc for the curve $(x(t),y(t),z(t),v(t),w(t))$ in $(G,d)$ is equal to the 
	length of corresponding arc for its projection $(x(t),y(t))$  on the Minkowski plane. $z=v=w=0$. One can easily see that projections 
	of indicated curves are metric straight lines on the Minkowski plane. Therefore the curves itself are metric straight lines.
\end{proof}

\begin{remark}
	The metric straight lines are obtained only in the case of Proposition \ref{norm1}, in particular, Proposition \ref{ageod1}.	
\end{remark}

{\bf 2.2.} Let us consider the case $\varphi_4=\varphi_5=0$, $\varphi_3\neq 0$.

\begin{proposition}
	\label{q2}
	Let $(x,y,z,v,w)(t)$, $t\in\mathbb{R}$, be an extremal with conditions
	$x(0)=y(0)=z(0)=v(0)=w(0)=0$ on the Cartan group such that $\varphi_4=\varphi_5=0$, $\varphi_3\neq 0$. Then the functions $\theta(t),$ $h(t)=(h_1(t),h_2(t)),$ $x(t),$ $y(t)$ are periodic with joint period $L=2S_0/|\varphi_3|,$ where $S_0$ is the area of the figure $U^{\ast}.$ The 
	projection $(x,y)(t)$ of the extremal onto the Minkowski plane $z=v=w=0$ with the quasinorm $F$ has a form
	\begin{equation}
	\label{xy}
	x(t)= \frac{h_2(t)-\varphi_2}{\varphi_3},\quad  y(t)=- \frac{h_1(t)-\varphi_1}{\varphi_3},
	\end{equation} 
	and it is a parametrized by the arc length periodic curve on an isoperimetrix.  In addition,
	$h_1=h_1(\theta(t)),$ $h_2=h_2(\theta(t))$ are given by formulas (\ref{h}), $\theta=\theta(t)$ is the inverse function to the function  $t(\theta)= \int_{\theta_0}^{\theta}(r^2(\xi)/\varphi_3)d\xi,$ and
	$$z(t)=\frac{t-\varphi_1x(t)-\varphi_2y(t)}{2\varphi_3}$$
	is equal to oriented area on the Euclidean plane with the Cartesian coordinates $x$, $y$,  traced by rectilinear segment connecting the origin with the point $(x(\tau),y(\tau))$, $\tau\in [0,t]$. The functions $v(t),$ $w(t)$ are defined by formulas
	(\ref{aa5}), (\ref{aa6}) or (\ref{vw}).
\end{proposition}

\begin{proof}
	The statements on the function $\theta(t)$ follow from (\ref{derst}). 
	It follows from (\ref{area}) and (\ref{h345}) that analogously to the second Kepler law the radius-vector-function  
	$h(\tau)=(h_1(\tau),h_2(\tau))\in U^{\ast},$ $t_1\leq \tau\leq t_2,$ traces in the plane $h_1,h_2$ (or, if it is desired, $u_1,u_2$ or $x,y$) with the standard Euclidean metric the oriented area $(\varphi_3/2)(t_2-t_1).$ Consequently, $h(t),$ $t\in\mathbb{R},$ is a periodic function with period $L=2S_0/|\varphi_3|,$ where $S_0$ is the area of the figure  $U^{\ast}.$ Moreover, (\ref{h345}), (\ref{a7}) and (\ref{a4}) imply
	formulas (\ref{xy}), i.e. the projection $(x,y)(t)$ of the curve $(x,y,z,v,w)(t)$  lies on the boundary $I(\varphi_1,\varphi_2,\varphi_3)$ of
	the figure obtained by rotation of $U^{\ast}/|\varphi_3|$ by the angle $\frac{\pi}{2}$ around the center 
	(origin of coordinates) with subsequent shift by vector $\left(-\frac{\varphi_2}{\varphi_3},\frac{\varphi_1}{\varphi_3}\right)$. 
	Thus, analogously to the case of the Heisenberg group with left-invariant sub-Finsler metric, considered in \cite{Ber2}, $I(\varphi_1,\varphi_2,\varphi_3)$ is an {\it isoperimetrix of the Minkowski plane with the quasinorm $F$} \cite{Lei}.
	
	Analogously to \cite{Ber2}, it follows from (\ref{xy}) that $(x(t),y(t))$  is a periodic curve on  $I(\varphi_1,\varphi_2,\varphi_3)$  with period $L$ indicated above.  It follows from (\ref{eq}) and (\ref{xy}) that
	\begin{equation}
	\label{z} 
	z(t)=\frac{t-\varphi_1x(t)-\varphi_2y(t)}{2\varphi_3}=\frac{1}{2\varphi_3^2}\left(\varphi_3t-\varphi_1h_2(t)+\varphi_2h_1(t)\right),
	\end{equation}
	\begin{equation}
	\label{zl}
	\quad z(L)=\frac{L}{2\varphi_3}=\frac{S_0}{|\varphi_3|\varphi_3}.
	\end{equation}
	The statement of Proposition \ref{q2} on the function $z(t)$ follows from (\ref{a4}). 
	Since $(x(t),y(t))$ lies on the isoperimetrix passing clockwise (counterclockwise) if  $\varphi_3<0$ ($\varphi_3>0$), then $z(t)$
	is a monotone function. In particular, $z(L)$ is the oriented area of the figure spanned by the isoperimetrix  $I(\varphi_1,\varphi_2,\varphi_3)$, or, what is the same, the area of $U^{\ast}/|\varphi_3|$ taken with the sign equal to the sign of  $z(L)$.
	
	The last statement was proved in Theorem \ref{mt}. 
	\end{proof}

{\bf 2.3.} Assume that $\varphi_4^2+\varphi_5^2\neq 0$.

\begin{lemma}
	\label{tconst1}
	If $\varphi_5\neq 0$ and the function $\theta(t)$ is constant on some non-degenerate interval $J\subset \mathbb{R}$,  then on $J$
	\begin{equation}
	\label{f1}
	x(t)=x_0+\frac{\varphi_5}{\mathcal{E}}(t-t_0),\,\,y(t)=-\frac{\varphi_4}{\mathcal{E}}(t-t_0)-\frac{1}{\varphi_5}(\varphi_3+\varphi_4x_0),\,\,z(t)=z_0+\frac{\varphi_3}{2\mathcal{E}}(t-t_0),
    \end{equation}
	\begin{equation}
	\label{f2}
	v(t)=v_0-\frac{\varphi_3\varphi_5}{12\mathcal{E}^2}(t-t_0)^2+\frac{1}{12\mathcal{E}}\left(\varphi_3x_0-6\varphi_5z_0\right)(t-t_0),
	\end{equation}
	\begin{equation}
	\label{f3}
	w(t)=w_0+\frac{\varphi_3\varphi_4}{12\mathcal{E}^2}(t-t_0)^2+\frac{1}{12\varphi_5\mathcal{E}}
	\left(6\varphi_4\varphi_5z_0-3\varphi_3\varphi_4x_0-\varphi_3^2\right)(t-t_0),
	\end{equation}
	where $x_0=x(t_0)$, $z_0=z(t_0)$, $v_0=v(t_0)$, $w_0=w(t_0)$, $t_0$ is a point of the interval $J$ closest to zero, $\mathcal{E}$ is the Casimir function (\ref{c12}). Moreover, $\mathcal{E}\neq 0,$ $F(\varphi_5/\mathcal{E},-\varphi_4/\mathcal{E})=1$, 
	$w_0$ is calculated by $x_0$, $y_0=-(\varphi_3+\varphi_4x_0)/\varphi_5$, $z_0$, $v_0$ and  (\ref{eq}) for $t=t_0$.
	
    In particular, for $\varphi_3=0$, we have $\mathcal{E}=\varphi_1\varphi_5-\varphi_2\varphi_4$ and 
	\begin{equation}
	\label{f01}
	x(t)=x_0+\frac{\varphi_5}{\mathcal{E}}(t-t_0),\quad
	y(t)=-\frac{\varphi_4}{\mathcal{E}}(t-t_0)-\frac{\varphi_4x_0}{\varphi_5},\quad
	z(t)=z_0,
	\end{equation}
    \begin{equation}
    \label{f02}
    v(t)=v_0-\frac{\varphi_5z_0}{2\mathcal{E}}(t-t_0),\quad
    	w(t)=w_0+\frac{\varphi_4z_0}{2\mathcal{E}}(t-t_0). 
    \end{equation}
\end{lemma}

\begin{proof}
If $\varphi_5\neq 0$ and $\theta(t)\equiv\theta_0$ on some non-degenerate interval  $J$, then $\dot{\theta}(t)\equiv 0$ and, in consequence of (\ref{derst}) and (\ref{a4}),  
\begin{equation}
\label{u2}
y(t)=-\frac{1}{\varphi_5}\left(\varphi_3+\varphi_4x(t)\right),\,\,z(t)=\frac{\varphi_3}{2\varphi_5}(x(t)-x_0)+z_0,\quad t\in J.
\end{equation}
It follows from the first equation (\ref{a4_1}) and (\ref{a4}) that
$$\dot{v}(t)=-\frac{\varphi_3}{6\varphi_5}x(t)u_1(t)+\left(\frac{\varphi_3x_0}{4\varphi_5}-\frac{z_0}{2}\right)u_1(t),$$
\begin{equation}
\label{vt}
v(t)=-\frac{\varphi_3}{12\varphi_5}\left(x^2(t)-x_0^2\right)+\left(\frac{\varphi_3x_0}{4\varphi_5}-\frac{z_0}{2}\right)(x(t)-x_0)+v_0,\,\,t\in J.
\end{equation}
The second equation (\ref{a4_1}), (\ref{a4}), and (\ref{u2}) imply that
$$\dot{w}(t)=\frac{\varphi_3\varphi_4}{6\varphi_5^2}x(t)u_1(t)+\left(\frac{\varphi_4 z_0}{2\varphi_5}-
\frac{\varphi_3^2}{12\varphi_5^2}-\frac{\varphi_3\varphi_4}{4\varphi_5^2}x_0\right)u_1(t),$$
\begin{equation}
\label{w0}
w(t)=\frac{\varphi_3\varphi_4}{12\varphi_5^2}\left(x^2(t)-x_0^2\right)+\left(\frac{\varphi_4 z_0}{2\varphi_5}-
\frac{\varphi_3^2}{12\varphi_5^2}-\frac{\varphi_3\varphi_4}{4\varphi_5^2}x_0\right)(x(t)-x_0)+w_0,\,\,t\in J.
\end{equation}

Inserting (\ref{u2}) -- (\ref{w0}) into the equality (\ref{eq}), we obtain
$$(\varphi_3^2+2\varphi_1\varphi_5-2\varphi_2\varphi_4)x(t)+3\varphi_3\varphi_5z_0+6\varphi_4\varphi_5v_0+6\varphi_5^2w_0-\varphi_3^2x_0-2\varphi_2\varphi_3=2\varphi_5t,\,\,t\in J.$$
This, (\ref{c12}), and the first equality in (\ref{a4}) imply that $\mathcal{E}\neq 0$ and
$u_1(t)=\varphi_5/\mathcal{E},$ thus $x(t)=\frac{\varphi_5}{\mathcal{E}}(t-t_0)+x_0$. Inserting the equality into (\ref{u2}) -- (\ref{w0}), we get (\ref{f1}) -- (\ref{f3}). Now it is easy to obtain the remaining statements.
\end{proof}

\begin{corollary}
	\label{rem1}
If $\varphi_5\neq 0$, the function  $\theta(t)$ is constant on some non-degenerate interval $J\subset \mathbb{R},$ and $0\in J$, then on $J,$ $(x,y,z,v,w)(t)$ is an extremal  (\ref{anorm2}) if $\varphi_4=0$, or an extremal (\ref{anorm}) if $\varphi_4\neq 0$.
\end{corollary}

\begin{proof}
In this case, it is more convenient to assume that $t_0=0$ under conditions of Lemma \ref{tconst1}. Then $x_0=y_0=z_0=v_0=w_0=0$ and $\varphi_3=0$ on the ground of (\ref{derst}). Inserting these equalities and the equality $\mathcal{E}=sF(s\varphi_5,-s\varphi_4),$ $s={\rm sgn}(\mathcal{E}),$ into (\ref{f01}) and (\ref{f02}), we get the required statement.
\end{proof}

\begin{lemma}
	\label{tconst2}
	If $\varphi_5=0$, $\varphi_4\neq 0$ and the function $\theta(t)$ is constant on some non-degenerate interval $J\subset \mathbb{R}$, then on $J$
\begin{equation}
\label{ccon1}
x(t)\equiv -\frac{\varphi_3}{\varphi_4},\quad y(t)=y_0-\frac{\varphi_4}{\mathcal{E}}(t-t_0),\quad
z(t)=z_0+\frac{\varphi_3}{2\mathcal{E}}(t-t_0),
\end{equation}
\begin{equation}
\label{ccon2}
v(t)=v_0-\frac{\varphi_3^2}{12\varphi_4\mathcal{E}}(t-t_0),\quad 
w(t)=w_0+\frac{\varphi_3\varphi_4}{12\mathcal{E}^2}(t-t_0)^2+\frac{1}{12\mathcal{E}}
\left(\varphi_3y_0+6\varphi_4z_0\right)(t-t_0),
\end{equation}
where $y_0=y(t_0)$, $z_0=z(t_0)$, $v_0=v(t_0),$ $w_0=w(t_0)$, $t_0$ is a point of the interval $J$ closest to zero, $\mathcal{E}=\varphi_3^2/2-\varphi_2\varphi_4$ is the Casimir function (\ref{c12}). Moreover,
 $w_0$ is calculated by  $x_0=-\frac{\varphi_3}{\varphi_4},$ $y_0,$ $z_0$, $v_0$ and (\ref{eq}) for $t=t_0$.
 
In particular, for $\varphi_3=0$ we have $\mathcal{E}=-\varphi_2\varphi_4$ and
\begin{equation}
\label{f03}	
x(t)\equiv 0,\,\,y(t)=y_0-\frac{\varphi_4}{\mathcal{E}}(t-t_0),\,\,z(t)\equiv z_0,\,\,v(t)\equiv v_0,\,\, 
w(t)=w_0+\frac{\varphi_4z_0}{2\mathcal{E}}(t-t_0).	
\end{equation}	
\end{lemma}

\begin{proof}
If $\varphi_5=0$, $\varphi_4\neq 0$ and $\theta(t)\equiv\theta_0$ on some non-degenerate interval $J$, then 
$\dot{\theta}(t)\equiv 0$ and, due to (\ref{derst}) and (\ref{a4}),  
\begin{equation}
\label{uu2}
x(t)=-\frac{\varphi_3}{\varphi_4},\quad z(t)=-\frac{\varphi_3}{2\varphi_4}(y(t)-y_0)+z_0,\quad t\in J.
\end{equation}
It follows from the first equation (\ref{a4_1}) and (\ref{a4}) that
$$\dot{v}(t)=\frac{\varphi_3^2}{12\varphi_4^2}u_2(t),\quad \dot{w}(t)=\frac{\varphi_3}{6\varphi_4}y(t)u_2(t)-\left(\frac{z_0}{2}+\frac{\varphi_3y_0}{4\varphi_4}\right)u_2(t),$$
hence for $t\in J,$
\begin{equation}
\label{wv}	
v(t)=\frac{\varphi_3^2}{12\varphi_4^2}(y(t)-y_0)+v_0,\,\, w(t)=\frac{\varphi_3}{12\varphi_4}\left(y^2(t)-y_0^2\right)-\left(\frac{z_0}{2}+\frac{\varphi_3y_0}{4\varphi_4}\right)\left(y(t)-y_0\right)+w_0.
\end{equation}

Inserting the equalities (\ref{uu2}), (\ref{wv}) into (\ref{eq}), we obtain
$$-(\varphi_3^2-2\varphi_2\varphi_4)y(t)-2\varphi_1\varphi_3+6\varphi_4^2v_0+\varphi_3^2y_0+3\varphi_3\varphi_4z_0=2\varphi_4t.$$
This, (\ref{c12}), and the second equation in (\ref{a4}) imply that $\mathcal{E}\neq 0$ and $u_2(t)=-\varphi_4/\mathcal{E}$, thus
$y(t)=y_0-\frac{\varphi_4}{\mathcal{E}}(t-t_0)$. Inserting the equality into (\ref{uu2}) and (\ref{wv}), we get the third equality (\ref{ccon1}) and (\ref{ccon2}). Now it is easy to obtain the remaining statements. 
\end{proof}

\begin{corollary}
	\label{rem2}
	If $\varphi_5=0$, $\varphi_4\neq 0$, the function $\theta(t)$ is constant on some non-degenerate interval  $J\subset \mathbb{R},$ and $0\in J$, then the trajectory	$(x,y,z,v,w)(t)$ on the interval $J$ is the extremal  (\ref{anorm1}).
\end{corollary}

\begin{proof}
In this case, it is more convenient to assume that $t_0=0$  under conditions of Lemma \ref{tconst2}. Then $x_0=y_0=z_0=v_0=w_0=0$ and $\varphi_3=0$ on the ground of (\ref{derst}). Inserting these equalities and the equality $\mathcal{E}=sF(0,-s\varphi_4),$ $s={\rm sgn}(\mathcal{E}),$ into (\ref{f03}), we get the required statement.
\end{proof}

Using (\ref{h}), the equality (\ref{dt}) can be rewritten as
\begin{equation}
\label{ttt}
r^4(\theta)\dot{\theta}^2=\varphi_3^2+2\sqrt{\varphi_4^2+\varphi_5^2}\left(r(\theta)\sin(\theta+\theta^{\ast})-
r(\theta_0)\sin(\theta_0+\theta^{\ast})\right),
\end{equation}
where
\begin{equation}
\label{tild}
\cos\theta^{\ast}=\frac{\varphi_4}{\sqrt{\varphi_4^2+\varphi_5^2}},\quad
\sin\theta^{\ast}=-\frac{\varphi_5}{\sqrt{\varphi_4^2+\varphi_5^2}}.
\end{equation}

Let $\tilde{F}$ be a quasinorm on $D(e)$, defined by the equality
$$\tilde{F}(u_1,u_2)=F(u_1\cos\theta^{\ast}+u_2\sin\theta^{\ast},-u_1\sin\theta^{\ast}+u_2\cos\theta^{\ast}),\,\,(u_1,u_2)\in D(e),$$
$\tilde{U}=\{u\in D(e)|\tilde{F}(u)\leq 1\}$ is the unit ball of the quasinorm $\tilde{F}$. It follows from the definitions of $\tilde{F}$  and its support Minkowski function that the curve $F_{\tilde{U}}(x,y)=1$, that is the polar boundary for the body $\tilde{U}$, is obtained from the curve $F_U(x,y)=1$ with the rotation by angle $\theta^{\ast}$ around the origin. Then $\tilde{r}(\theta)=r(\theta-\theta^{\ast})$, $\theta\in\mathbb{R}$, is a polar equation of the curve $F_{\tilde{U}}(x,y)=1$. Set $\tilde{\theta}(t)=\theta(t)+\theta^{\ast}$ and $\tilde{\theta}_0=\tilde{\theta}(0)=\theta_0+\theta^{\ast}$.
Then the equation (\ref{ttt}) can be rewritten as
\begin{equation}
\label{dts}
\left(\dot{\tilde{\theta}}\right)^{2}=\frac{\varphi_3^2+2\sqrt{\varphi_4^2+\varphi_5^2}\left(\tilde{r}(\tilde{\theta})\sin(\tilde{\theta})-\tilde{r}(\tilde{\theta}_0)
	\sin(\tilde{\theta}_0)\right)}{\tilde{r}^4(\tilde{\theta})}.
\end{equation}	

\begin{theorem}
	\label{mt1}	
	If $\varphi_4^2+\varphi_5^2\neq 0$ then any extremal on the Cartan group starting at the unit is defined by the equations
	(\ref{xxt}), (\ref{yyt}) (with arbitrary measureable integrands of indicated view and continuously differentiable function $\theta=\theta(t)$  satisfying  (\ref{derst}), (\ref{dt})).
	
	Moreover, if $\varphi_5\neq 0$ then
	\begin{equation}
	\label{zt}
	z(t)=-\frac{1}{\varphi_5}\left(\varphi_2+\left(\varphi_3+\frac{1}{2}\varphi_4x(t)+\frac{1}{2}\varphi_5y(t)\right)x(t)-
	r(\theta(t))\sin\theta(t)\right),
	\end{equation} 
	the function $v(t)$ is given by the second formula (\ref{vw}), and
	\begin{equation}
	\label{wt}
	w(t)=\frac{1}{3\varphi_5}\left(t-\varphi_1x(t)-\varphi_2y(t)-2\varphi_3z(t)-3\varphi_4v(t)-\frac{\varphi_4}{2}x(t)z(t)-
	\frac{\varphi_5}{2}y(t)z(t)\right).
	\end{equation}
	
	If $\varphi_5=0$ and $\varphi_4\neq 0$ then
	$$z(t)=\frac{1}{\varphi_4}\left(\varphi_1-\left(\varphi_3+\frac{1}{2}\varphi_4x(t)\right)y(t)-r(\theta(t))
	\cos\theta(t)\right),$$
	$$v(t)=\frac{1}{3\varphi_4}\left(t-\varphi_1x(t)-\varphi_2y(t)-2\varphi_3z(t)-\frac{\varphi_4}{2}x(t)z(t)\right)$$
	and the function $w(t)$ is given by the second formula (\ref{vw}).
	
	Let us set
	$$\theta_0:=\theta(0),\quad \mathcal{E}_0=\max_{h\in U^{\ast}}(\varphi_5h_1-\varphi_4h_2),\quad 
	\mathcal{E}_{-1}=\min_{h\in U^{\ast}}(\varphi_5h_1-\varphi_4h_2).$$ 
	The following cases are possible.
	
	1.  Let  $\varphi_3\neq 0$ и $\mathcal{E}>\mathcal{E}_0$. Then the function  $\theta(t),$ $t\in\mathbb{R}$, is inverse to the function $t(\theta)$, defined by formula
	\begin{equation}
	\label{tth} t(\theta)=\int_{\theta_0}^{\theta}\frac{r^2(\xi)d\xi}{\varphi_3\sqrt{1+(2\varphi_4/\varphi_3^2)(r(\xi)\sin\xi-\varphi_2)-
    (2\varphi_5/\varphi_3^2)(r(\xi)\cos\xi-\varphi_1)}}, 
	\end{equation}
	where
	\begin{equation}
	\label{f12}
	r(\theta_0)\cos\theta_0=\varphi_1,\quad r(\theta_0)\sin\theta_0=\varphi_2.
\end{equation}
	
	2. Let  $\varphi_3=0$ and $\mathcal{E}=\mathcal{E}_{-1}$. Then $\theta(t)\equiv\theta_0$  and the desired extremal is the metric straight line  (\ref{anorm1}), (\ref{anorm2}) or (\ref{anorm}).
	
	3. Let $\mathcal{E}_{-1}< \mathcal{E}<\mathcal{E}_0$. Then we have for some numbers $t_1$, $t_2$, $t_1<t_2$, for any $t\in\mathbb{R}$ and $k\in\mathbb{Z}$  
	\begin{equation}
	\label{dts00}
	\theta(t+2k(t_2-t_1))=\theta(t),\quad\dot{\theta}(t_i+t)=-\dot{\theta}(t_i-t),\,\,\theta(t_i+t)=\theta(t_i-t),\,\,i=1,2.
	\end{equation}
	
	3.1. If  $\varphi_3\neq 0$ then $t_i=t(\theta_i)$, $i=1,2$, in equalities (\ref{dts00}) are calculated by (\ref{tth}), where $\theta_1\neq\theta_2$ are the nearest to  $\theta_0$ values such that the right-hand side in (\ref{dt})  vanishes and $\varphi_3(\theta_2-\theta_1)>0$.
	
	3.2. If $\varphi_3=0$ then $\theta_2\neq\theta_1=\theta_0$, $t_1=0$, $t_2=t(\theta_2)$ in equalities (\ref{dts00}), where
	\begin{equation}
	\label{tth0}
	t(\theta)=\pm\int\limits_{\theta_0}^{\theta}
	\frac{r^2(\xi)d\xi}{\sqrt{2\varphi_4( r(\xi)\sin{\xi}-\varphi_2)-2\varphi_5( r(\xi)\cos{\xi}-\varphi_1)}},
	\end{equation}
	on the right--hand side stands $+$ (respectively, $-$) if $\theta_2>\theta_0$ ($\theta_2<\theta_0$) and (\ref{f12}) holds. 	
	Here $\theta_2\neq\theta_0$ is a number such that  
	$\varphi_4\left(h_2(\theta)-h_2(\theta_0)\right)\geq\varphi_5\left(h_1(\theta)-h_1(\theta_0)\right)$
	for any $\theta$ from interval $I=\left(\min(\theta_0,\theta_1),\max(\theta_0,\theta_1)\right)$, 
	and the equality holds only for $\theta=\theta_0$ and $\theta=\theta_2$.
	
	4. Let  $\varphi_3\neq 0$ and $\mathcal{E}=\mathcal{E}_0$. Then there exist the nearest to  $\theta_0$ values
	$\theta_1$, $\theta_2$ such that $\theta_1<\theta_0<\theta_2$ and the right-hand side in (\ref{dt}) vanishes for $\theta=\theta_i,$ $i=1,2.$
	
	If improper integrals (\ref{tth}) diverge for $\theta=\theta_1$ and $\theta=\theta_2$, then  $\theta(t)\in (\theta_1,\theta_2),$ $t\in \mathbb{R},$ is the inverse function to the function $t(\theta)$ defined by (\ref{tth}). 
	
	If  improper integral (\ref{tth}) is finite for $\theta=\theta_i$, $i\in\{1,2\}$, then  the function $\theta(t)$ is not unique and can take constant values equal to $\theta_1+2\pi k$ for some $k\in\mathbb{Z}$ (and with an arbitrary alternation of increase's and decrease's intervals) if $i=1,2$ and $\theta_2=\theta_1+ 2\pi,$ and equal to
	$\theta_i$ in other cases, on some non-degenerate closed intervals of arbitrary length, on which 
	(\ref{f1}) -- (\ref{f3})  are valid if $\varphi_5\neq 0$ or (\ref{ccon1}), (\ref{ccon2}) are valid if $\varphi_5=0$.
			
	5. Let  $\varphi_3=0$ and $\mathcal{E}=\mathcal{E}_{0}$. Then there exists the largest segment
	$[\theta_1,\theta_2]$, $\theta_1\leq\theta_2$,  such that $\theta_0\in [\theta_1,\theta_2]$ and
	$\varphi_5(h_1(\theta)-\varphi_1)=\varphi_4(h_2(\theta)-\varphi_2)$ 
	for any $\theta\in [\theta_1,\theta_2]$. If $\theta_0=\theta_2$ (respectively, $\theta_0=\theta_1$) we will assume that $t(\theta)$ is an improper integral (\ref{tth0}) for $\theta\in[\theta_0,\theta_1+2\pi]$ (respectively, $\theta\in[\theta_1-2\pi,\theta_0]$) without $+$ and $-$.
	
	Then $\theta(t)\equiv\theta_0$ and the desired extremal is the metric straight line (\ref{anorm1}), (\ref{anorm2}) or (\ref{anorm}) in the following cases:
	
	5.1. $\theta_1=\theta_0=\theta_2$ and $t(\theta)=\infty$ for $\theta\nearrow\theta_0$ and for $\theta\searrow\theta_0$;
	
	5.2. $\theta_1<\theta_0<\theta_2$;
	
	5.3. $\theta_0=\theta_1<\theta_2$ and $t(\theta)=\infty$ for $\theta\nearrow\theta_0$;
	
	5.4. $\theta_0=\theta_2>\theta_1$ and $t(\theta)=\infty$ for $\theta\searrow\theta_0$.
	
	In all other cases, the function $\theta(t)$ is not unique and can take constant values on some closed intervals of arbitrary length, on which (\ref{f01}), (\ref{f02}) are valid if $\varphi_5\neq 0$, or (\ref{f03}) is valid if $\varphi_5=0$. These constant values may be equal to 1) $\theta_0$ and $\theta_1+2\pi$ if  $\theta_1<\theta_0=\theta_2$, $t(\theta)$ is finite for $\theta\searrow\theta_0,$ and $t(\theta_1+2\pi)$ is finite; 2) $\theta_0$ and $\theta_2-2\pi$ if  $\theta_0=\theta_1<\theta_2$, $t(\theta)$ is finite for $\theta\nearrow\theta_0,$ and $t(\theta_2-2\pi)$ is finite; 3) $\theta_0+2\pi k$ for some $k\in\mathbb{Z}$ (and with an arbitrary alternation of increase's and decrease's intervals) if $\theta_0=\theta_1=\theta_2$,  $t(\theta)$ is finite for $\theta\nearrow\theta_0$ and for $\theta\searrow\theta_0$; 4) $\theta_0$ in all other cases.
	\end{theorem}

\begin{proof}
	The first statement of this Theorem follows from Theorem \ref{mt}. Moreover, formulae for $z(t)$ are consequences of equalities (\ref{h12}) and (\ref{h}),
	formulae for $w(t)$ in the case $\varphi_5\neq 0$ and for $v(t)$ in the case $\varphi_5=0$ follow directly from (\ref{eq}).
	Formulae (\ref{tth}) and (\ref{tth0}) follow from the equality (\ref{dt}), which can be written in the form (\ref{dts}).
		
	All other statements of Theorem \ref{mt1} follow from Lemmas \ref{tconst1} and \ref{tconst2} of our paper and from Theorem 2 and its proof in paper 	\cite{BerZub} for the following replacements in the last theorem:
	$$\varphi_1\,\Rightarrow\,\frac{\varphi_1\varphi_4+\varphi_2\varphi_5}{\sqrt{\varphi_4^2+\varphi_5^2}},\quad
	\varphi_2\,\Rightarrow\,\frac{\varphi_2\varphi_4-\varphi_1\varphi_5}{\sqrt{\varphi_4^2+\varphi_5^2}},\quad
	\varphi_3\,\Rightarrow\,\varphi_3,\quad\varphi_4\,\Rightarrow\,\sqrt{\varphi_4^2+\varphi_5^2},$$
	$$h_1(\theta)\,\Rightarrow\,\frac{\varphi_4h_1(\theta)+\varphi_5h_2(\theta)}{\sqrt{\varphi_4^2+\varphi_5^2}},\quad h_2(\theta)\,\Rightarrow\,\frac{\varphi_4h_2(\theta)-\varphi_5h_1(\theta)}{\sqrt{\varphi_4^2+\varphi_5^2}},
	\quad\mathcal{E}\,\Rightarrow\,\mathcal{E}.$$	
\end{proof}

\end{document}